\documentclass[12pt, reqno]{amsart}
\usepackage{longtable}
\PassOptionsToPackage{table}{xcolor}
\usepackage{stmaryrd}
\expandafter\def\csname opt@stmaryrd.sty\endcsname
{only,shortleftarrow,shortrightarrow}
\usepackage{extpfeil}
\usepackage{amsmath,amssymb,amsthm}
\usepackage{hyperref}
\hypersetup{colorlinks=true,urlcolor=blue,citecolor=blue,linkcolor=blue}
\usepackage{courier}
\usepackage{colonequals}
\usepackage{tikz}
\usepackage{tikz-cd}
\usetikzlibrary{calc,matrix,arrows,decorations.markings}
\usepackage{array}
\usepackage{color}
\usepackage{enumerate}
\usepackage{nicefrac}
\usepackage{listings}
\usepackage{siunitx}
\usepackage{seqsplit}
\usepackage{cancel}
\usepackage{fullpage}

\newcommand{\Q}{\mathbb{Q}}
\newcommand{\PP}{\mathbb{P}}

\newcommand{\C}{\mathbb{C}}

\newcommand{\Z}{\mathbb{Z}}

\newcommand{\new}{\operatorname{new}}
\newcommand{\Jac}{\operatorname{Jac}}

\DeclareFontFamily{U}{wncy}{}
    \DeclareFontShape{U}{wncy}{m}{n}{<->wncyr10}{}
    \DeclareSymbolFont{mcy}{U}{wncy}{m}{n}
    \DeclareMathSymbol{\Sh}{\mathord}{mcy}{"58} 

\newcommand{\defi}[1]{\textsf{\textit{#1}}}

\newcommand{\psmod}[1]{~(\textup{\text{mod}}~{#1})}

\usepackage{xcolor}

\newenvironment{enumalph}
{\begin{enumerate}}
	{\end{enumerate}}

\newenvironment{enumalg}
{\begin{enumerate}}
{\end{enumerate}}

\newenvironment{enumalgii}
{\begin{enumerate}}
{\end{enumerate}}

\newenvironment{enumroman}
{\begin{enumerate}}
	{\end{enumerate}}

\numberwithin{equation}{section}

\newtheorem{theorem}[equation]{Theorem}

\newtheorem{prop}[equation]{Proposition}
\newtheorem{proposition}[equation]{Proposition}
\newtheorem{lem}[equation]{Lemma}
\newtheorem{lemma}[equation]{Lemma}
\newtheorem{cor}[equation]{Corollary}
\newtheorem{corollary}[equation]{Corollary}

\theoremstyle{definition}

\theoremstyle{remark}
\newtheorem{remark}[equation]{Remark}

\begin{document}

\author{Oana Padurariu}
\address{Oana Padurariu \\
Max-Planck-Institut für Mathematik Bonn\\
Germany}
\urladdr{https://sites.google.com/view/oanapadurariu/home}
\email{opadurariu@mpim-bonn.mpg.de}

\author{Sun Woo Park}
\address{Sun Woo Park \\
Max-Planck-Institut für Mathematik Bonn\\
Germany}
\urladdr{https://sites.google.com/wisc.edu/spark483}
\email{s.park@mpim-bonn.mpg.de}

\author{John Voight}
\address{John Voight \\ University of Sydney \\ Australia
}
\urladdr{https://jvoight.github.io}
\email{jvoight@gmail.com}

\title[Gain of primes of good reduction]{Unexpected primes of good reduction in quotients of modular and Shimura curves}

\begin{abstract}
We find all zero-dimensional spaces of newforms of weight $2$ and squarefree level $N$ with a fixed Atkin--Lehner sign pattern.  As an application, we classify unexpected primes of good reduction of Atkin--Lehner quotients of modular and Shimura curves of squarefree levels.
\end{abstract}

\maketitle

\section{Introduction}

Modular curves of genus zero, including their Atkin--Lehner quotients, have played a significant role in arithmetic geometry and even in group theory through monstrous moonshine.  Since the genus of a modular curve is the dimension of the space of cuspforms of weight $2$, in this paper we are led to the more general question: which spaces of newforms \emph{with fixed Atkin--Lehner signs} are zero?  

Let $N$ be squarefree and let $W_N \colonequals \{w_d : d \mid N\}$ denote its Atkin--Lehner group, the elementary abelian $2$-group with cardinality $\#W_N = 2^{\omega(N)}$.  A \defi{sign pattern} for $N$ is a character
\[
  \varepsilon_N \colon W_N \to \{\pm 1\},
\]
determined by its values on $w_p$ for $p \mid N$.  Let
\[  S_2^{\new,\varepsilon_N}(N)
  \colonequals \{ f \in S_2^{\new}(\Gamma_0(N)) : w_d f=\varepsilon_N(d)f\text{ for all }d\mid N\} \]
be the Atkin--Lehner $\varepsilon_N$-eigenspace over $\C$ of classical newforms of level $N$ and trivial character.  We call the pair $(N,\varepsilon_N)$ \defi{lacking} if $S_2^{\new,\varepsilon_N}(N)=\{0\}$.  

A basic dimension estimate already quickly shows that there are only finitely many lacking pairs.  Indeed, the main term in the trace formula for $\dim S_2^{\new,\varepsilon_N}(N)$ grows like $N/2^{\omega(N)}$, whereas the Atkin--Lehner trace terms are controlled by class numbers of imaginary quadratic fields and have size roughly $O(\sqrt{N}\log N)$ after summing over divisors.  Thus every sign pattern occurs once $N$ is sufficiently large.  

Our first result is to make this effective and then compute the complete list.

\begin{theorem} \label{theorem:newforms}
Let $N$ be squarefree and let $\varepsilon_N$ be a sign pattern for $N$.  Then the pairs $(N,\varepsilon_N)$ for which
$S_2^{\new,\varepsilon_N}(N)=\{0\}$ are precisely the pairs listed in Tables \textup{\ref{Complete_N_5}}, \textup{\ref{Complete_N_3}}, \textup{\ref{Complete_N_2}}, and \textup{\ref{Complete_N_2.1}}.
\end{theorem}

The proof of Theorem \ref{theorem:newforms} occupies Section \ref{sec:dimensions}; we use a version of the trace formula adapted for this purpose by Martin \cite{Martin18}.  Weinstein \cite{Weinstein09} proves a result in the same spirit for Hilbert modular forms with prescribed local ramification: up to finitely many exceptions, all prescribed local behaviours occur.  However, prescribing the Atkin--Lehner signs creates smaller spaces, with additional class-number terms in the trace formula; controlling these terms uniformly is the main analytic task here, and then evaluating the remaining finite range is the main computational task.

Our motivation for Theorem~\ref{theorem:newforms} comes from the following example.  The curve $X_0(194)$ has bad reduction at $2$ but its Atkin--Lehner quotient $X_0^*(194) = X_0(194)/W_{194}$ of genus $3$ has \emph{good} reduction at $2$, and therefore so does its Jacobian!  Our second result classifies all such unexpected (or gained) primes of good reduction for Atkin--Lehner quotients of Shimura curves of squarefree level.  

\begin{theorem} \label{theorem:modularshimura}
Let $D,N$ be coprime squarefree integers with $\omega(D)$ even (allowing $D=1$), and let $X_0^*(D;N) \colonequals X_0(D;N)/W_{DN}$ be the full Atkin--Lehner quotient of the modular or Shimura curve of discriminant $D$ and Eichler level $N$.  Then the following statements hold.
\begin{enumalph}
\item The curve $X_0^*(D;N)$ has good reduction at $p$ if and only if its Jacobian has good reduction at $p$.  
\item If $X_0^*(D;N)$ has genus $>0$ and $p \mid D$, then $X_0^*(D;N)$ has bad reduction at $p$.  
\item The triples $(D,N,p)$ for primes $p \mid DN$ such that $g(X_0^*(D;N))>0$ and $X_0^*(D;N)$ has good reduction at $p$ are precisely those listed in Tables \textup{\ref{CM_points_table}} and \textup{\ref{Shimura_points_table}}.
\end{enumalph}
\end{theorem}

Indeed, the Jacobian $\Jac(X_0^*(D;N))$ decomposes up to isogeny over $\Q$ indexed by newforms of level dividing $DN$ that are new at $D$ with Atkin--Lehner eigenvalues specified by the sign pattern $\varepsilon$ for $DN$ with $\varepsilon(w_p)=-1,+1$ according as $p \mid D$ or $p \mid N$.  Section \ref{sec:gained-primes} explains this criterion and then deduces Theorem~\ref{theorem:modularshimura} using Theorem~\ref{theorem:newforms}.

The calculations in this paper deliberately focus on the cleanest case, the one relevant to our application: squarefree level, weight $2$, and trivial character.  Several natural extensions should be possible.  The first is to allow arbitrary (nonsquarefree) level: there are only additional technicalities, with a more complicated trace formula and more local bookkeeping.  The second would be to allow arbitrary nebentypus character, including signs at primes where the local component of the character $\chi$ is at most quadratic.  Third, we could allow higher weight $k \geq 3$; the analysis is similar and a bit easier, with main term now the larger $(k-1)(N/2^{\omega(N)})$.  Finally, one can ask for the analogous result for Hilbert modular forms over totally real fields, including a decomposition according to class group characters.  

All computations implemented in this paper are performed in Magma \cite{MR1484478}. We used a database of class numbers of imaginary quadratic extensions from LMFDB \cite{lmfdb} and Kimball Martin's implementation of trace formula \cite{Martin18}.  We refer to the Github repository \cite{PPV_github} for all related codes, computation time, and outputs.

\subsection*{Acknowledgements}

The authors would like to thank Kimball Martin for helpful conversations and for the hospitality of the Lodha Mathematical Sciences Institute (LMSI) where some of this work was performed. We are also grateful for the mathematical discussions with Nikola Adžaga and Freddy Saia. Voight was supported by the Simons Foundation (550029) and (SFI-MPS-Infra\-structure-00008650). Padurariu and Park would like to thank Max Planck Institute for Mathematics for their hospitality and computational resources.

\section{Dimensions of weight 2 newforms}
\label{sec:dimensions}

In this section, we prove Theorem~\ref{theorem:newforms}.  Let $N\in\Z_{\geq 1}$ be squarefree and let $\varepsilon_N$ be a sign pattern for $N$.  For $n\in\Z_{\geq 1}$, let $\omega(n)$ be the number of distinct prime factors of $n$, and let $\omega'(n)$ be the number of distinct odd prime factors of $n$.

\subsection*{Dimension formula}

We use explicit formulas for dimensions of newspaces with prescribed Atkin--Lehner signs from Martin \cite{Martin18}.  For $d>0$ squarefree, let $\Delta_d$ be the discriminant of $K_d \colonequals \Q(\sqrt{-d})$, and let $h'(\Delta_d)$ be its modified class number, the usual class number of the ring of integers unless $d=1,3$, where instead $h'(\Delta_1)=1/2$ and $h'(\Delta_3)=1/3$.  

\begin{prop} \label{prop:asplit}
We have
\begin{align} \label{eqn:martin-lower-bound}
\begin{split}
    \dim S_{2}^{\new, \varepsilon_N}(N) \geq \frac{1}{2^{\omega(N)}} \biggl[ &\dim S_2^{\new}(N) - 2^{\omega(N)} \\
    &- \delta[2 \mid N]   \left(h'(\Delta_2)   |b(2,N/2)|   2^{\omega(N/2)-1} + \frac{1}{2} 2^{\omega(N/2)} \right) \\
    &- \delta[3 \mid N]   \left(h'(\Delta_3)   |b(3,N/3)|   2^{\omega(N/3)-1} + \frac{2}{3} 2^{\omega(N/3)}  \right) \\ 
    &- \frac{1}{2}\sum_{\substack{d > 3 \\ d \mid N}} |b(d,N/d)|   h'(\Delta_d)   2^{\omega'(N/d)} \biggr],
\end{split}
\end{align}
where $\delta[P]=1$ if $P$ is true and $0$ otherwise, and
\begin{equation} \label{eqn:bmn}
    b(m,n) \colonequals \begin{cases}
    1, &\text{ if } m \equiv 1,2,5,6 \pmod{8} \text{ and } n \equiv 1 \pmod{2}; \\
    -1, &\text{ if } m \equiv 1,2,5,6 \pmod{8} \text{ and } n \equiv 0 \pmod{2}; \\
    4, &\text{ if } m \equiv 3 \pmod{8} \text{ and } n \equiv 1 \pmod{2}; \\
    -2, &\text{ if } m \equiv 3 \pmod{8} \text{ and } n \equiv 0 \pmod{2}; \\
    2, &\text{ if } m \equiv 7 \pmod{8} \text{ and } n \equiv 1 \pmod{2}; \\
    0, &\text{ if } m \equiv 7 \pmod{8} \text{ and } n \equiv 0 \pmod{2}.
    \end{cases}
\end{equation}
\end{prop}

\begin{proof}
Combine Martin \cite[Proposition 1.4, Proposition 3.2, Theorem 3.3]{Martin18}, taking absolute values of the trace terms which can appear with either sign.
\end{proof}

The main term in Proposition~\ref{prop:asplit} is the dimension of the full newspace.

\begin{lem} \label{lem:s2nnew}
We have
\[
  \dim S_2^{\new}(N)=\frac{\varphi(N)}{12}-\frac{e_4}{4}-\frac{e_3}{3}+\mu(N),
\]
where $e_D=\prod_{p\mid N}\bigl(\left(\frac{-D}{p}\right)-1\bigr)$ for $D=3,4$, and $\mu$ is the M\"obius function.
\end{lem}

\begin{proof}
See Martin \cite[Theorem 2.1]{Martin18} (citing \cite[Theorem 4]{Martin05}).
\end{proof}

\subsection*{Estimates}

We now isolate the estimates used to make the finite computation effective.

\begin{lem} \label{lem:coarse-lower-bound}
Let $m=\omega(N)$.  Define
\begin{equation} \label{eqn:A-def}
A(N) \colonequals \begin{cases}
        \frac{19}{12}2^m+1, &\text{ if } N \equiv \pm 1 \pmod{6}; \\
        \frac{9}{4}2^m+1, &\text{ if } N \equiv 3 \pmod{6}; \\
        3\cdot 2^m+1, &\text{ if } N \equiv 0 \pmod{2}.
    \end{cases}
\end{equation}
Then for every sign pattern $\varepsilon_N$,
\begin{equation} \label{eqn:coarse-lower-bound}
\dim S_{2}^{\new,\varepsilon_N}(N)
\geq
\frac{1}{2^m}\left(
\frac{\varphi(N)}{12}-A(N)
-\frac{2}{\pi}\sum_{\substack{d>1\\d\mid N}}
\sqrt d\log(4d)2^{\omega(N/d)}
\right).
\end{equation}
\end{lem}

\begin{proof}
Substitute Lemma~\ref{lem:s2nnew} into Proposition~\ref{prop:asplit}.  Since $\mu(N)\geq -1$ and
\begin{equation}
  \left|e_D\right|
  =\prod_{p\mid N}\left|\left(\frac{-D}{p}\right)-1\right|
  \leq 2^m
\end{equation}
the contribution from the full newspace formula, together with the term $-2^m$ in Proposition~\ref{prop:asplit}, is bounded below by
\begin{equation}
  \frac{\varphi(N)}{12}-\left(1+\frac14+\frac13\right)2^m-1
  =\frac{\varphi(N)}{12}-\frac{19}{12}2^m-1.
\end{equation}
This gives the first case of $A(N)$, because then neither the $d=2$ nor the $d=3$ exceptional term occurs.

If $N\equiv 3\pmod 6$, then $3\mid N$ and $2\nmid N$.  The $d=3$ term contributes at most
\[
  h'(\Delta_3)|b(3,N/3)|2^{m-2}+\frac23 2^{m-1}
  =\frac13\cdot 4\cdot 2^{m-2}+\frac23 2^{m-1}
  =\frac23 2^m,
\]
using $h'(\Delta_3)=1/3$ and $|b(3,N/3)|=4$.  Adding this to $\frac{19}{12}2^m+1$ gives $\frac94 2^m+1$.  Similarly, if $N$ is even then the $d=2$ term contributes at most $2^{m-1}$ since $h'(\Delta_2)=1$ and $|b(2,N/2)|=1$; if $3 \mid N$ then the $d=3$ term contributes at most 
\[
  h'(\Delta_3)|b(3,N/3)|2^{m-2}+\frac23 2^{m-1}
  =\frac13\cdot 2\cdot 2^{m-2}+\frac23 2^{m-1}
  =\frac12 2^m.
\]
If $3\nmid N$, there is no contribution.  Thus the total coefficient of $2^m$ is at most
\[
  \frac{19}{12}+\frac12+\frac12=\frac{31}{12}<3,
\]
which gives the stated even case.

To control the remaining sum, from \eqref{eqn:bmn}, we have $|b(d,N/d)|\leq 4$; also $\omega'(N/d)\leq \omega(N/d)$.  The Louboutin bound \cite[Corollary 2]{Lou04}
\begin{equation} \label{eqn:louboutin-bound}
  h(\Delta_d)=h'(\Delta_d)\leq \frac{\sqrt d}{\pi}\log(4d)
\end{equation}
for $d>3$ therefore gives
\begin{equation}
  \frac12 |b(d,N/d)|h'(\Delta_d)2^{\omega'(N/d)}
  \leq \frac{2}{\pi}\sqrt d\log(4d)2^{\omega(N/d)}.
\end{equation}
Extending the resulting sum from $d>3$ to $d>1$ only increases the error term, so \eqref{eqn:coarse-lower-bound} follows.
\end{proof}

\begin{lem} \label{lem:divisor-sum-bound}
Let $m=\omega(N)$ and put
\[
  \kappa(N)\colonequals
  \begin{cases}
    1, &\text{if } N\equiv \pm1\pmod 6;\\
    2, &\text{otherwise.}
  \end{cases}
\]
Then
\begin{equation} \label{eqn:divisor-sum-bound}
\sum_{\substack{d>1\\d\mid N}}
\sqrt d\log(4d)2^{\omega(N/d)}
\leq
\kappa(N)2^{m-1}\sqrt N\left(\log N+4\log 2\right).
\end{equation}
\end{lem}

\begin{proof}
Fix $j$ with $0\leq j\leq m-1$, and sum over divisors $d\mid N$ with $d>1$ and $\omega(N/d)=j$.  Write $e=N/d$.  Then $e$ is squarefree and has $j$ prime factors.

We first prove
\begin{equation} \label{eqn:small-prime-factor-bound}
  2^j\sqrt d=\sqrt{4^j d}\leq \kappa(N)\sqrt N.
\end{equation}
If $N\equiv \pm1\pmod 6$, then every prime factor of $e$ is at least $5$, so $e \geq 5^j\geq 4^j$, and \eqref{eqn:small-prime-factor-bound} holds with $\kappa(N)=1$.  Otherwise $\kappa(N)=2$.  The smallest possible product of $j$ distinct primes is at least $4^{j-1}$.  Hence $e\geq 4^{j-1}$, so $4^j d=4^jN/e\leq 4N$, proving \eqref{eqn:small-prime-factor-bound}.

Let
\[
  \mathcal D_j\colonequals \{d\mid N: d>1,\ \omega(N/d)=j\}.
\]
By \eqref{eqn:small-prime-factor-bound},
\begin{equation} \label{eqn:j-sum-start}
\sum_{d\in\mathcal D_j}\sqrt d\log(4d)2^{\omega(N/d)}
\leq
\kappa(N)\sqrt N\sum_{d\in\mathcal D_j}\log(4d).
\end{equation}
The set $\mathcal D_j$ has $\binom{m}{j}$ elements.  Moreover, each prime factor of $N$ occurs in exactly $\binom{m-1}{j}$ of the divisors $d\in\mathcal D_j$, since it occurs in $d=N/e$ precisely when it is not chosen as one of the $j$ prime factors of $e$.  Therefore
\[
  \prod_{d\in\mathcal D_j} 4d
  =4^{\binom{m}{j}}N^{\binom{m-1}{j}},
\]
and hence
\begin{equation} \label{eqn:log-sum-Dj}
  \sum_{d\in\mathcal D_j}\log(4d)
  =\binom{m-1}{j}\log N+\binom{m}{j}\log 4.
\end{equation}
Combining \eqref{eqn:j-sum-start} and \eqref{eqn:log-sum-Dj}, then summing over $0\leq j\leq m-1$, gives
\begin{equation}
\begin{aligned}
\sum_{\substack{d>1\\d\mid N}}
\sqrt d\log(4d)2^{\omega(N/d)}
&\leq \kappa(N)\sqrt N
\left(
  \sum_{j=0}^{m-1}\binom{m-1}{j}\log N
  +\sum_{j=0}^{m-1}\binom{m}{j}\log 4
\right)\\
&\leq \kappa(N)\sqrt N
\left(2^{m-1}\log N+2^m\log 4\right)\\
&=\kappa(N)2^{m-1}\sqrt N\left(\log N+4\log 2\right),
\end{aligned}
\end{equation}
which is \eqref{eqn:divisor-sum-bound}.
\end{proof}

We then obtain the following positivity criterion. 

\begin{cor} \label{cor:positivity-criterion}
Let $m=\omega(N)$.  If
\begin{equation} \label{eqn:positivity-criterion}
  \frac{\varphi(N)}{12}
  >
  A(N)+\frac{\kappa(N)2^m}{\pi}\sqrt N\left(\log N+4\log2\right),
\end{equation}
then $\dim S_2^{\new,\varepsilon_N}(N)>0$ for every sign pattern $\varepsilon_N$.
\end{cor}

\begin{proof}
Combine Lemmas~\ref{lem:coarse-lower-bound} and~\ref{lem:divisor-sum-bound}.
\end{proof}

\subsection*{Applying the positivity criterion}

The upshot of \eqref{eqn:positivity-criterion} is that $A(N)$ and $\kappa(N)$ depend only on the number of prime divisors of $N$ and the congruence class $N$ modulo $6$; we exploit this in the following statement.

\begin{prop} \label{prop:omega-bound}
The following statements hold.
\begin{enumroman}
    \item If $N \equiv \pm 1 \pmod{6}$, then $\dim S_2^{\new, \varepsilon_N}(N)>0$ if $\omega(N) \geq 9$.
    \item If $N \equiv 3 \pmod{6}$, then $\dim S_2^{\new, \varepsilon_N}(N)>0$ if $\omega(N) \geq 11$.
    \item If $N \equiv 0 \pmod{2}$, then $\dim S_2^{\new, \varepsilon_N}(N)>0$ if $\omega(N) \geq 13$.
\end{enumroman}
\end{prop}

\begin{proof}
We prove that \eqref{eqn:positivity-criterion} holds in the indicated ranges.

We choose one of the three congruence classes.  We let $m \geq 1$ and consider integers $N$ in the given congruence class with $\omega(N)=m$.  Let $P_m$ be the product of the $m$ smallest possible prime factors of $N$ in that class: so for (i) we take the product of the first $m$ primes $p \geq 5$; for (ii) we take $3$ times the product of the first $m-1$ primes at least $5$; and for (iii) we just take the product of the first $m$ primes.  Let
\begin{equation}
  \rho_m \colonequals \prod_{p\mid P_m}\left(1-\frac1p\right).
\end{equation}
For any squarefree $N$ in the congruence class with $\omega(N)=m$, we have $N \geq P_m$ by construction and 
\begin{equation} \label{eqn:phi-min-bound}
  \varphi(N)=N\prod_{p\mid N}\left(1-\frac1p\right) \geq \rho_m N.
\end{equation}

For $x \geq 1$ define
\begin{equation} \label{eqn:fmx}
  F_{m}(x)
  \colonequals
  \frac{\rho_m}{12}x
  -A(m)
  -\frac{\kappa 2^m}{\pi}\sqrt x\left(\log x+4\log2\right),
\end{equation}
where $A(m)$ and $\kappa$ are determined by $m$ and the congruence class.  By \eqref{eqn:phi-min-bound}, the positivity criterion follows from $F_{m}(N)>0$.  We have
\begin{equation} \label{eqn:deriv}
  F'_{m}(x)
  =\frac{\rho_m}{12}
  -\frac{\kappa 2^m}{\pi}\frac{\log x+4\log2+2}{2\sqrt x}.
\end{equation}
The function $(\log x+4\log2+2)/(2\sqrt x)$ is decreasing for $x\geq 1$, thus to check $F_m(x)>0$ for $x \geq x_0$ it suffices to show that $F_m'(x_0) > 0$.  

We first check the base values $m=9,11,13$ for the three classes.  At the minimal possible $x=P_{\mathcal C,m}$, direct numerical evaluation gives
\[
\begin{array}{c|c||c|c|c}
\textup{congruence} & m & P_m & F_{m}(P_m) & F'_{m}(N) \\ \hline\hline
N\equiv \pm1\psmod6 & 9 & 33426748355 & >4.7\cdot 10^8 & >0.025,\\
N\equiv 3\psmod6 & 11 & 3710369067405 & >1.2\cdot 10^{10} & >0.013,\\
N\equiv 0\psmod2 & 13 & 304250263527210 & >3.9\cdot 10^{11} & >0.006.
\end{array}
\]
Thus $F_{m}(x)>0$ for all $x\geq P_m$ in these base cases.

It remains to propagate the result to larger $m$ which we do by induction on $m$.  Suppose \eqref{eqn:positivity-criterion} holds for every $N$ in one of the three classes with $\omega(N)=m$, where $m$ is at least the corresponding base value.  Let $N^*$ be in the same class with $\omega(N^*)=m+1$, and let $q$ be the largest prime factor of $N^*$.  Put $N=N^*/q$.  Then $N$ is still in the same class, and $q\geq 37$ in all three inductions.  By the induction hypothesis,
\[
  \frac{\varphi(N)}{12}
  >A(m)+\frac{\kappa 2^m}{\pi}\sqrt N(\log N+4\log2).
\]
Multiplying both sides by $q-1$ to get
\[ \frac{\varphi(N^*)}{12} > (q-1)A(m)+\frac{\kappa 2^m (q-1)}{\pi}\sqrt N(\log N+4\log2). \]
Of course $(q-1)A(m) \geq 2A(m) > A(m+1)$; and we claim that for $q\geq 37$,
\begin{equation} \label{eqn:q-growth}
  q-1\geq 2\sqrt q\left(1+\frac{\log q}{4\log2}\right)
\end{equation}
since the difference is $>7$ at $q=37$ and the derivative is positive for $q \geq 37$.  Since $\log N+4\log2\geq 4\log2$, \eqref{eqn:q-growth} implies
\[
  (q-1)(\log N+4\log2)
  \geq 2\sqrt q(\log(qN)+4\log2).
\]
Therefore
\begin{equation}
  (q-1)\frac{\kappa 2^m}{\pi}\sqrt N(\log N+4\log2)
  \geq
  \frac{\kappa 2^{m+1}}{\pi}\sqrt {N^*}(\log N^*+4\log2).
\end{equation}
Combining these two inequalities proves \eqref{eqn:positivity-criterion} for $N^*$.  The proof is then finished by induction.
\end{proof}

\subsection*{Computation}

After Proposition~\ref{prop:omega-bound}, we are left with finitely many possible values of $\omega(N)$ in each congruence class.  

\begin{lem}
For $N$ in the given congruence class with given number of prime divisors $\omega(N)$, if $N>B$ in Table~\textup{\ref{Lower_Bounds_N}}, then $\dim S_2^{\new,\varepsilon_N}(N)>0$.
\end{lem}

\begin{table}[ht]
\begin{tabular}{c||c|c|c}
$\omega(N)$ & $N \equiv \pm 1 \pmod{6}$ & $N \equiv 3 \pmod{6}$ & $N \equiv 0 \pmod{2}$ \\ \hline \hline
13 & - & - & - \\ \hline
12 & - & - & 52320255687514 \\ \hline
11 & - & - & 11302808177135 \\ \hline
10 & - & 584712825763 & 2402574532579 \\ \hline
9 & - & 122979855890 & 505520992240 \\ \hline
8 & 2544226303 & 25568947411 & 103627324753 \\ \hline
7 & 518227848 & 5170085802 & 20662946692 \\ \hline
6 & 102294401 & 1014636104 & 4029222669 \\ \hline
5 & 19508619 & 194205148 & 745913629 \\ \hline
4 & 3607458 & 35153388 & 131967432 \\ \hline
3 & 625485 & 6050500 & 20186499 \\ \hline
2 & 101953 & 892300 & 2588857 
\end{tabular}
\caption{Bounds $B$ such that $N>B$ implies $\dim S_2^{\new,\varepsilon_N}(N)>0$ for every sign pattern.}\label{Lower_Bounds_N}
\end{table}

A dash in Table~\ref{Lower_Bounds_N} means that Proposition~\ref{prop:omega-bound} already gives positivity for all such $N$.

\begin{proof}
The same positivity criterion gives explicit cutoffs for $N$ for fixed $m=\omega(N)$ and congruence class.  We compute the associated upper bounds on $N$ in the following table, where an entry $B$ means that every squarefree $N$ in that congruence class with the indicated value of $\omega(N)$ and with $N>B$ has $\dim S_2^{\new,\varepsilon_N}(N)>0$ for every sign pattern $\varepsilon_N$.  
The entries in Table~\ref{Lower_Bounds_N} are obtained by rounding up after solving $F_m(x) = 0$ where $F_m(x)$ is as in \eqref{eqn:fmx}, for fixed $m$ in a given congruence class.  A numerical solution can be quickly estimated, then $F_m(B)>0$ verified rigorously.  Computing the derivative and repeating the argument in the proof of Proposition~\ref{prop:omega-bound} then also shows that $F_m(N)>0$ for all $N\geq B$.  
\end{proof}

It remains to enumerate the finitely many levels below these cutoffs and then evaluate Martin's exact dimension formula. We recall from \cite[Proposition 1.4, 3.2]{Martin18} that
\begin{align*}
    \dim S_2^{new,\epsilon_N}(N) &= 2^{-\omega(N)} \sum_{d \mid N} \epsilon_N(d) \cdot \mathrm{tr}_{S_2^{new}(N)} W_d, \text{ where} \\
    \mathrm{tr}_{S_2^{new}(N)} W_d &:= -\frac{1}{2} h'(\Delta_d) b(d,N/d) \prod_{\substack{p \mid (N/d) \\ p \neq 2}} \left( \left(\frac{\Delta_d}{p} \right) - 1 \right) + (-1)^{\omega(N) - \omega(d)} \\
    & \hspace{30pt} - \delta[d = 2] \cdot \prod_{p \mid (N/d)} \left( \left( \frac{-4}{p} \right) - 1 \right) - \delta[d = 3] \cdot \prod_{p \mid (N/d)} \left( \left( \frac{-3}{p} \right) - 1 \right).
\end{align*}

The computation proceeds with the following steps.
\begin{enumalg}
\item Fix a row and column of Table~\ref{Lower_Bounds_N}.  Thus $m=\omega(N)$, the congruence class of $N$, and a bound $B$ are fixed.  If the table entry is a dash, there is nothing to enumerate.

\item Generate all increasing $m$-tuples of primes $(p_1,\ldots,p_m)$ in the required congruence class with $\prod_i p_i\leq B$.  During the recursion, after $p_1,\ldots,p_i$ have been chosen, it is enough to consider
\begin{equation} \label{eqn:tuple-recursion-bound}
  p_i<p_{i+1}<\left(\frac{B}{\prod_{k=1}^i p_k}\right)^{1/(m-i)},
\end{equation}
otherwise even choosing all remaining primes at least $p_{i+1}$ would make the product exceed $B$.

\item Apply the filter from the starting inequality: if $\dim S_2^{\new,\varepsilon_N}(N)=0$ for some sign pattern, then by \eqref{eqn:coarse-lower-bound} we must have
\begin{equation} \label{eqn:coarse-necessary-condition}
  \frac{\varphi(N)}{12}
  \leq
  A(N)+\frac{2}{\pi}\sum_{\substack{d>1\\d\mid N}}
  \sqrt d\log(4d)2^{m-\omega(d)}.
\end{equation}
We discard $N$ if \eqref{eqn:coarse-necessary-condition} fails.

\item Apply a sign-pattern filter.  For a fixed sign pattern $\varepsilon_N$, Martin's trace formula identifies exactly which divisors $d$ can contribute negatively.  Let $S(\varepsilon_N)$ be the set of divisors $d\mid N$ satisfying one of the following conditions:
\begin{enumalgii}
    \item $N$ is odd, $d>3$,
    \[
      \varepsilon_N(d)b(d,N/d)(-1)^{\omega(N/d)}=1,
    \]
    and $(-d\,|\,p)=-1$ for all $p\mid N/d$;

    \item $N$ is even, $d\equiv 1,3\pmod 4$, $d>3$,
    \[
      \varepsilon_N(d)b(d,N/d)(-1)^{\omega(N/d)}=-1,
    \]
    and $(-d\,|\,p)=-1$ for all primes $p\mid N/d$ with $p\neq 2$;

    \item $N$ is even, $d$ is even, $d>3$,
    \[
      \varepsilon_N(d)b(d,N/d)(-1)^{\omega(N/d)}=-1,
    \]
    and $(-d\,|\,p)=-1$ for all $p\mid N/d$.
\end{enumalgii}
Then a lacking pair satisfies the sharper necessary inequality
\begin{equation} \label{eqn:sign-filter-necessary-condition}
  \frac{\varphi(N)}{12}
  \leq 
  A(N)+\frac{2}{\pi}\sum_{d\in S(\varepsilon_N)}
  \sqrt d\log(4d)2^{\omega(N/d)}.
\end{equation}
This step discards sign patterns for which the negative trace terms are too small to cancel the main term.

\item For the remaining pairs $(N,\varepsilon_N)$, compute $\dim S_2^{\new,\varepsilon_N}(N)$ exactly using a Magma implementation \cite{MR1484478} of the Sage code \cite{sagemath} supplied by Martin \cite{Martin18}.
\end{enumalg}

Steps 2--4 are not logically necessary, but they really help---so much so that we can call the LMFDB \cite{lmfdb} for the class numbers we need by the time we get to the final step.

The resulting zero-dimensional eigenspaces are exactly those listed in Tables~\ref{Complete_N_5}, \textup{\ref{Complete_N_3}}, \ref{Complete_N_2}, and~\ref{Complete_N_2.1}.  The Magma enumeration and filtering in Steps 2--5 took 278234.59 seconds on a single 4 x Intel Xeon Gold 6148 CPU with 64 threads, with total memory usage 33901.19MB.

As an immediate corollary, we obtain the squarefree composite levels for which no weight $2$ newform has root number $-1$.

\begin{corollary} \label{cor:root-number-minus-list}
Let $S_2^{\new,-}(N)$ denote the subspace generated by newforms of weight $2$ and squarefree level $N$ with root number $-1$.  Then there are exactly $45$ squarefree composite integers $N$ such that $\dim S_2^{\new,-}(N)=0$, namely
\begin{align*}
&6,10,14,15,21,22,26,30,33,34,35,38,39,42,46,51,55,62,66,69,70,74,78,86,87,94,95,\\
&105,110,111,114,134,146,159,174,182,186,194,195,206,222,230,231,255,266.
\end{align*}
\end{corollary}

\begin{proof}
For a weight $k$ newform of squarefree level $N$ and Atkin--Lehner sign pattern $\varepsilon_N$, the root number is
\[
  (-1)^{k/2}\prod_{p\mid N}\varepsilon_N(p).
\]
For $k=2$, the root number is $-1$ exactly when $\prod_{p\mid N}\varepsilon_N(p)=1$, equivalently $\varepsilon_N(p)=-1$ for an even number of prime divisors $p\mid N$.  Therefore $S_2^{\new,-}(N)=0$ exactly when all sign eigenspaces in Tables~\ref{Complete_N_5}, \ref{Complete_N_3}, \ref{Complete_N_2}, and~\ref{Complete_N_2.1} with this parity condition are zero.  Reading off those levels gives the displayed list.
\end{proof}

\section{Gained primes of good reduction}
\label{sec:gained-primes}

As in the introduction, let $D,N$ be coprime squarefree integers with $D$ the product of an even number of primes (with $D=1$ allowed).  Let $X_0(D;N)$ be the modular or Shimura curve associated to an Eichler order of level $N$ in a quaternion algebra of discriminant $D$, and let $X_0^*(D;N) \colonequals X_0(D;N)/W_{DN}$ be the full Atkin--Lehner quotient of the associated modular or Shimura curve.  We say that a prime $p \mid DN$ is an \defi{unexpected} or \defi{gained} prime of good reduction if the curve $X_0^*(D;N)$ or its Jacobian $J_0^*(D;N) \colonequals \Jac(X_0^*(D;N))$ has good reduction at $p$.  

In the case when the genus of $X_0^*(D;N)$ is $g=0$, the Jacobian is trivial and automatically all primes $p \mid DN$ are gained primes of good reduction.  Whenever $X_0^*(D;N)(\Q) \simeq \PP^1$---for example whenever $D=1$ and the cusp $\infty$ is $\Q$-rational, or more generally when there is a rational CM point---then again all primes are gained. This is always the case, as proven by \cite[Proposition 4.2]{NR15}.

\subsection*{Semistable quotients and a one-vertex criterion}

Let $R$ be a discrete valuation ring with fraction field $K$ and residue field $k$.  Let $X$ be a smooth, proper, geometrically connected curve over $K$ with semistable reduction, and let $\mathcal X$ be a proper semistable model over $R$.  We write $\Gamma_{\mathcal X}$ for the dual graph of the geometric special fibre $\mathcal X_{\bar k}$: its vertices are the irreducible components, its edges are the nodes, and a self-intersection on a component gives a loop.

\begin{lemma} \label{lem:raynaud-dual-graph}
Let $J \colonequals \Jac(X)$, and let $\mathcal J$ be its N\'eron model over $R$.  Then the identity component of the geometric special fibre fits in an extension
\[
  0 \to T \to \mathcal J^0_{\bar k} \to A \to 0,
\]
where $T$ is a torus and $A$ is an abelian variety, and there is a canonical isomorphism of the character group of $T$ with $H_1(\Gamma_{\mathcal X},\Z)$.

In particular, $J$ has good reduction if and only if $\Gamma_{\mathcal X}$ is a tree, equivalently if and only if the special fibre is of compact type.
\end{lemma}

\begin{proof}
This is Raynaud's description of the relative Picard functor; see Raynaud \cite{Raynaud70} and Bosch--L\"utkebohmert--Raynaud \cite[Section~9.2, Theorem~8]{BLR90}.  If the total space is not regular, resolving a node only subdivides the corresponding edge and therefore does not change $H_1$ of the dual graph.  The final assertion is the usual compact-type criterion for good reduction of the Jacobian.
\end{proof}

\begin{corollary} \label{cor:one-vertex-criterion}
Suppose that $X$ has a semistable model whose geometric special fibre is irreducible.  Then the following are equivalent:
\begin{enumroman}
\item $X$ has good reduction;
\item the dual graph has no loops; and
\item $\Jac(X)$ has good reduction.
\end{enumroman}
\end{corollary}

\begin{proof}
The dual graph has one vertex, so every edge is a loop and $b_1(\Gamma_{\mathcal X})=\#E(\Gamma_{\mathcal X})$.  Thus it is a tree if and only if it has no edges.  Having no edges means that the irreducible special fibre has no nodes, hence is smooth.  The result now follows from Lemma~\ref{lem:raynaud-dual-graph}.
\end{proof}

We next verify that this criterion applies to the full Atkin--Lehner quotient.

\begin{proposition} \label{prop:one-component-quotient}
Let $p\mid N$.  Then the curve $X_0^*(D;N)$ has semistable reduction at $p$, and it has a semistable model over $\Z_p$ whose geometric special fibre is irreducible.  In particular, $X_0^*(D;N)$ has good reduction at $p$ if and only if $J_0^*(D;N)$ has good reduction at $p$.  
\end{proposition}

\begin{proof}
The integral model of $X_0(D;N)$ at $p$ has two geometric irreducible components, each identified with the corresponding curve with the $p$-part of the level removed, and these components meet transversally at the supersingular points: for modular curves this is the Deligne--Rapoport model \cite[Chapter~V, Section~1]{DeligneRapoport73}; see also Mazur \cite[Appendix]{Mazur77}.  For Shimura curves over $\Q$, the analogous statement is Buzzard \cite[Theorem~4.7]{Buzzard97}; see also Carayol \cite{Carayol86}.  In both cases the involution $w_p$ interchanges the two components.

Let $\mathcal X$ be the semistable integral model of
$X_0(D;N)$ over $\Z_p$ described above, and form the finite quotient
$ \pi\colon \mathcal X \to \mathcal X/W_{DN}$.  The quotient $\mathcal X/W_{DN}$ is a normal proper flat curve over $\Z_p$, with generic fibre $X_0^*(D;N)$.  By
Liu--Lorenzini~\cite[Proposition~1.6]{LiuLorenzini1999},
the curve $\mathcal Y$ is semistable.

As above, the special fibre of $\mathcal X$ has two irreducible components with $w_p$ interchanging them.  Since $\pi$ is finite and surjective, the images cover the special fibre, so the special fiber is irreducible.  

The final statement follows from Corollary~\ref{cor:one-vertex-criterion}.
\end{proof}

\subsection*{Automorphic criterion}

By the Jacquet--Langlands correspondence, the space of regular differentials on $X_0(D;N)$ is 
\begin{equation} \label{eqn:D-new-decomposition}
H^0(X_0(D;N),\Omega^1) \simeq 
S_2^{\textup{$D$-new}}(\Gamma_0(DN)) \simeq \bigoplus_{D \mid M \mid DN} S_2^{\textup{new}}(\Gamma_0(M))^{\oplus d(DN/M)}
\end{equation}
decomposing into newspaces with multiplicity according to Atkin--Lehner theory.  For $M$ with $D \mid M \mid DN$, define the sign pattern $\varepsilon_M$ by $\varepsilon_M(w_q)=-1,+1$ according as $q \mid D$ or $q \mid (M/D)$ (signs flip at ramified primes).  

Taking $W_{DN}$-invariants removes the oldform multiplicities in \eqref{eqn:D-new-decomposition}.  Indeed, for each prime $q\mid (DN/M)$ (necessarily $q\mid N$), the $q$-old degeneracy space is two-dimensional and $w_q$ has one-dimensional eigenspaces for each sign.  Selecting the invariant line at every such $q$ leaves one copy of the prescribed eigenspace at level $M$; see Atkin--Li \cite{AL78}.  Consequently
\begin{equation}\label{eqn:S2Dnew}
  H^0(X^*_0(D;N),\Omega^1)
  \simeq
  S_2^{\textup{$D$-new},\varepsilon_{DN}}(\Gamma_0(DN))
  \simeq
  \bigoplus_{D\mid M\mid DN}
  S_2^{\new,\varepsilon_M}(\Gamma_0(M)).
\end{equation}
It follows then that the isotypic decomposition of $J^*_0(D;N)$ up to isogeny is 
\begin{equation}\label{eqn:Afmf}
  J^*_0(D;N) \sim_{\Q}
  \prod_{D\mid M\mid DN}
  \prod_{[f]} A_f,
\end{equation}
where $[f]$ runs over Galois orbits of normalized newforms in $S_2^{\new,\varepsilon_M}(\Gamma_0(M))$ and $A_f$ is the associated modular abelian variety.  

We obtain the following criterion, extending Proposition~\ref{prop:one-component-quotient}.  

\begin{lemma}\label{lem:good-reduction-newforms}
$J_0^*(D;N)$ has good reduction at $p \mid DN$ if and only if 
\[ \dim S_2^{\textup{new},\varepsilon_M}(\Gamma_0(M)) = 0 \]
for all $M$ such that $D \mid M \mid DN$ and $p \mid M$.  
\end{lemma}

In other words, unexpected good reduction for the Jacobian occurs exactly when all newspaces in the decomposition with level a multiple of $p$ are lacking.  

\begin{proof}
For $\ell\neq p$, the $\ell$-adic Tate module of $A$ is unramified at $p$ if and only if the same is true for each $A_f$ occurring in the isogeny decomposition.  By the N\'eron--Ogg--Shafarevich criterion \cite{SerreTate68}, this is equivalent to good reduction.  For a modular abelian variety attached to a newform of squarefree level $M$, ramification at $p$ occurs exactly when $p\mid M$.
\end{proof}

\begin{cor} \label{cor:X*J*}
If $X_0^*(D;N)$ has genus $>0$, then both $X_0^*(D;N)$ and $J_0^*(D;N)$ have bad reduction at all primes $p \mid D$.
\end{cor}

\begin{proof}
Since the sum \eqref{eqn:Afmf} is nonzero, any form in level $M$ will have $p \mid M$ so by the previous lemma, $p$ is a prime of bad reduction.
\end{proof}

\begin{proof}[{Proof of Theorem~\ref{theorem:modularshimura}}]
Part (a) is proven in Proposition~\ref{prop:one-component-quotient}; part (b) is proven in Corollary~\ref{cor:X*J*}.  So we prove part (c).  Lemma~\ref{lem:good-reduction-newforms} gives a straightforward way to compute Tables \ref{CM_points_table} and \ref{Shimura_points_table} using Tables \ref{Complete_N_5}--\ref{Complete_N_2.1}.  Namely, for each squarefree level $N'$ in the finite range produced by Theorem~\ref{theorem:newforms}, and for all possible decompositions $N'=DN$ with $D$ the product of an even number of primes, we get a sign pattern and determine the set of primes $p$ which satisfy Lemma~\ref{lem:good-reduction-newforms}.  Keeping exactly those pairs $(N,p)$ gives Table \ref{CM_points_table} and \ref{Shimura_points_table}.
\end{proof}

\begin{remark}
    We note that unexpected primes of good reduction can appear in covers of $X_0^*(N)$ and $X_0^*(D;N)$. 
    \begin{itemize}
        \item A modular curve example is $X_0(74)/\langle w_{74} \rangle$, a genus $2$ curve of conductor $37^2$. In this case, the quotient $X_0^*(74)$ is a genus $1$ curve of conductor $37$.
        \item A Shimura curve example can be seen in \cite[Example 5.1]{PS25}. The curve $$X_0(210;11) / \langle w_2,w_5,w_7,w_{33} \rangle$$ is a genus $1$ curve with unexpected prime of good reduction at $11$. The quotient $X_0^*(210; 11)$ is a genus $0$ curve.
    \end{itemize}
\end{remark}

\section{Tables}

We provide Tables \ref{Complete_N_5}, \ref{Complete_N_3}, \ref{Complete_N_2}, and \ref{Complete_N_2.1} from Section \ref{sec:dimensions}, and Tables \ref{CM_points_table} and \ref{Shimura_points_table} from Section \ref{sec:gained-primes}. We elaborate on the entries of Tables \ref{CM_points_table} and \ref{Shimura_points_table}. The first column (respectively the first two columns) indicates the level of the modular curve. The next two columns indicate the genera of curves $X_0(N)$ and $X_0^*(N)$ (and respectively $X_0(D;N)$ and $X_0^*(D;N)$). We use abbreviations $g \colonequals g(X_0(N))$ (and respectively $g \colonequals g(X_0(D;N))$) and $g^* \colonequals g(X_0^*(N))$ (and respectively $g^* \colonequals g(X_0^*(D;N))$). The fourth (respectively the fifth) column indicates gained primes of good reduction. We can find such good primes by looking at the isotypic decomposition of the Jacobians of respective modular curves (abbreviated as $J^*$). By Proposition \ref{prop:one-component-quotient}, such primes are also gained primes of good reduction for the modular curves (abbreviated as $X^*$).  Some increments of primes of good reduction for modular Atkin--Lehner quotients $X_0(N)^*$ have been studied in previous literature, which we explain in the last column of Table \ref{CM_points_table}. We omit the genus-zero quotients, since their Jacobians are trivial and therefore do not distinguish the reduction behaviour of the curve.





\begin{small}
\begin{table}[ht!]
\begin{tabular}{c|c}
Sign pattern $\varepsilon_N$ & Values of $N$ \\ \hline \hline
$[1,1]$ & 35, 55, 95, 119 \\ \hline
$[-1,-1]$ & 35, 55, 65, 85, 95, 115, 215, 77, 119, 161, 143\\ \hline \hline
$[1,-1,-1]$ & 455 \\ \hline
$[-1,-1,1]$ & 455
\end{tabular}
\caption{Complete list of squarefree integers $N \equiv \pm 1 \psmod{6}$ such that $\dim S_2^{\textup{new}, \varepsilon_N(N)} = 0$ with respective sign patterns.} \label{Complete_N_5} 
\end{table}

\begin{table}[ht!]
\begin{tabular}{c|c}
Sign pattern $\varepsilon_N$ & Values of $N$ \\ \hline \hline
$[1,1]$ & 15, 21, 33, 39, 51, 69, 87, 111, 159 \\ \hline
$[1,-1]$ & 21, 57, 93\\ \hline
$[-1,1]$ & 15, 33\\ \hline
$[-1,-1]$ & 15, 21, 33, 39, 51, 57, 69, 87, 93, 111, 129, 159, 183, 237\\ \hline \hline
$[1,1,1]$ & 105, 195, 255, 231 \\ \hline
$[1,1,-1]$ & 165, 273 \\ \hline
$[1,-1,1]$ & 105, 165 \\ \hline
$[1,-1,-1]$ & 105, 165, 195, 255, 285, 435, 615, 231 \\ \hline
$[-1,1,1]$ & 105 \\ \hline
$[-1,1,-1]$ & 105, 165, 195, 255, 231, 273, 399 \\ \hline
$[-1,-1,1]$ & 105, 165, 195, 255, 285, 435, 615, 231, 273, 399 
\end{tabular}
\caption{Squarefree integers $N \equiv 3 \psmod{6}$ such that $\dim S_2^{\textup{new}, \varepsilon_N(N)} = 0$ with respective sign patterns.} \label{Complete_N_3} 
\end{table}

\newpage

\begin{table}[ht]
\begin{tabular}{c|c}
Sign pattern $\varepsilon_N$ & Values of $N$ \\ \hline \hline
$[1,1]$ & 6, 10, 14, 22, 26, 34, 38, 46, 62, 74, \\
& 86, 94, 134, 146, 194, 206 \\ \hline
$[1,-1]$ & 6, 10, 22, 34, 58, 82\\ \hline
$[-1,1]$ & 6, 10, 14, 22, 46\\ \hline
$[-1,-1]$ & 6, 10, 14, 22, 26, 34, 38, 46, 58, 62, \\
& 74, 82, 86, 94, 106, 118, 122, 134, 146, \\
& 166, 178, 194, 202, 206, 314\\ \hline \hline
$[1,1,1]$ & 30, 42, 66, 78, 114, 174, 186, 222, 366, 654, \\
& 70, 110, 130, 170, 230, 530, 182, 266 \\ \hline
$[1,1,-1]$ & 30, 42, 66, 102, 138, 282, 498, 70, 130, 190, 310, 322, 418 \\ \hline
$[1,-1,1]$ & 42, 78, 114, 258, 70, 130, 154 \\ \hline
$[1,-1,-1]$ & 30, 42, 66, 78, 102, 114, 138, 174, 186, 222, \\
& 246, 282, 318, 354, 426, 534, 70, 110, 190, 230, 290, \\
& 310, 154, 182, 238, 266, 434, 518, 602, 286, 374, 494 \\ \hline
$[-1,1,1]$ & 30, 78, 102, 190 \\ \hline
$[-1,1,-1]$ & 30, 42, 66, 78, 102, 114, 138, 174, 186, 222, \\
& 246, 258, 402, 426, 474, 70, 110, 130, 170, 230, 290, \\
& 370, 154, 182, 238, 266, 434, 374, 494 \\ \hline
$[-1,-1,1]$ & 30, 42, 66, 78, 102, 114, 138, 174, 186, 222, \\
& 246, 258, 282, 318, 354, 366, 402, 426, 438, 474, 498, \\
& 534, 582, 642, 654, 678, 822, 70, 110, 130, 170, 190, \\
& 230, 290, 310, 370, 410, 530, 610, 154, 182, 266, 322, \\
& 406, 518, 602, 286, 374, 418, 442 \\ \hline
$[-1,-1,-1]$ & 30, 42, 78, 70
\end{tabular}
\caption{Squarefree integers $N \equiv 0 \psmod{2}$ such that $\dim S_2^{\textup{new}, \varepsilon_N(N)} = 0$ with respective sign patterns, $\omega(N) \leq 3$.}\label{Complete_N_2} 
\end{table}

\newpage

\begin{table}[ht]
\begin{tabular}{c|c}
Sign pattern $\varepsilon_N$ & Values of $N$ \\ \hline \hline
$[1,1,1,1]$ & 330, 510, 870, 1230, 546, 770 \\ \hline 
$[1,1,1,-1]$ & 210, 390, 1110, 1290, 858 \\ \hline 
$[1,1,-1,1]$ & 210, 330, 570, 462 \\ \hline 
$[1,1,-1,-1]$ & 210, 390, 510, 690, 1590, 546, 798, 770 \\ \hline 
$[1,-1,1,1]$ & 210, 330, 510, 690, 870, 1122 \\ \hline 
$[1,-1,1,-1]$ & 210, 330, 390, 510, 570, 930, 462, 714, 910 \\ \hline 
$[1,-1,-1,1]$ & 210, 330, 390, 510, 690, 1110, 462, 546, 714, 966, \\
& 1218, 858, 1326, 1190 \\ \hline
$[1,-1,-1,-1]$ & 330 \\ \hline 
$[-1,1,1,1]$ & 210, 462, 798, 910 \\ \hline
$[-1,1,1,-1]$ & 210, 330, 390, 570, 1770, 462, 546, 966, 770 \\ \hline
$[-1,1,-1,1]$ & 210, 330, 390, 510, 570, 690, 930, 1410, \\
& 546, 714, 798, 1302, 1190 \\ \hline
$[-1,1,-1,-1]$ & 462 \\ \hline 
$[-1,-1,1,1]$ & 210, 330, 390, 510, 570, 690, 870, 930, \\
& 1410, 462, 546, 714, 966, 1218, 858, 770 \\ \hline 
$[-1,-1,1,-1]$ & 330 \\ \hline 
$[-1,-1,-1,-1]$ & 210, 330, 390, 510, 570, 690, 870, 930, \\
& 1110, 1230, 1290, 1410, 1590, 1770, 2010, 2130, 462, \\
& 546, 714, 798, 966, 1218, 1302, 1722, 1974, 858, 1122, \\
& 1254, 1914, 1326, 1794, 1938, 770, 910, 1190 \\ \hline\hline
$[1,1,-1,1,-1]$ & 2310 \\ \hline 
$[1,-1,1,1,-1]$ & 2730 \\ \hline
$[1,-1,-1,-1,-1]$ & 2310, 2730, 3990 \\ \hline 
$[-1,1,1,1,-1]$ & 3570 \\ \hline 
$[-1,1,1,-1,1]$ & 2310 \\ \hline 
$[-1,1,-1,-1,-1]$ & 2310, 4290 \\ \hline 
$[-1,-1,1,-1,-1]$ & 2310, 2730 \\ \hline 
$[-1,-1,-1,1,-1]$ & 2310, 2730, 3570, 3990 \\ \hline
$[-1,-1,-1,-1,1]$ & 2310, 2730, 3570, 3990, 4290 
\end{tabular}
\caption{Squarefree integers $N \equiv 0 \psmod{2}$ such that $\dim S_2^{\textup{new}, \varepsilon_N(N)} = 0$ with respective sign patterns, $\omega(N)=4,5$.}\label{Complete_N_2.1} 
\end{table}

\newpage

\begin{table}[ht]
\begin{tabular}{c|c|c||c|c}
$N$ & $g$ & $g^*$ & $p$ good for $J^*$ and $X^*$ & exceptional $\simeq$ \\ \hline\hline
$74$ & $8$ & $1$ & $2$ & $X_0(74)^* \simeq X_0(34)^*$ \\ \hline
$86$ & $10$ & $1$ & $2$ & $X_0(86)^* \simeq X_0(43)^*$\\ \hline
$111$ & $11$ & $1$ & $3$ & $X_0(111)^* \simeq X_0(37)^*$\\ \hline
$114$ & $17$ & $1$ & $2$ & $X_0(114)^* \simeq X_0(57)^*$\\ \hline
$130$ & $17$ & $1$ & $2$ & N/A\\ \hline
$134$ & $16$ & $2$ & $2$ & $X_0(134)^* \simeq X_0(67)^*$ \\ \hline
$146$ & $17$ & $2$ & $2$ & $X_0(146)^* \simeq X_0(73)^*$\\ \hline
$159$ & $17$ & $1$ & $3$ & $X_0(159)^* \simeq X_0(53)^*$ \\ \hline
$170$ & $23$ & $2$ & $2$ & N/A\\ \hline
$174$ & $27$ & $1$ & $3$ & $X_0(174)^* \simeq X_0(58)^*$ \\ \hline
$182$ & $25$ & $1$ & $2$ & $X_0(182)^* \simeq X_0(91)^*$\\ \hline
$186$ & $29$ & $2$ & $2$ & $X_0(186)^* \simeq X_0(93)^*$\\ \hline
$194$ & $23$ & $3$ & $2$ & $X_0(194)^* \simeq X_0(97)^*$ \\ \hline
$195$ & $25$ & $1$ & $3$ & N/A \\ \hline
$206$ & $25$ & $2$ & $2$ & $X_0(206)^* \simeq X_0(103)^*$ \\ \hline
$222$ & $35$ & $1$ & $2,3$ & $X_0(222)^* \simeq X_0(37)^*$\\ \hline
$230$ & $33$ & $2$ & $2$ & $X_0(230)^* \simeq X_0(115)^*$ \\ \hline
$231$ & $29$ & $1$ & $3$ & $X_0(231)^* \simeq X_0(77)^*$ \\ \hline
$255$ & $33$ & $2$ & $3$ & N/A \\ \hline
$266$ & $37$ & $2$ & $2$ & $X_0(266)^* \simeq X_0(133)^*$ \\ \hline
$330$ & $65$ & $2$ & $2$ & N/A \\ \hline
$546$ & $105$ & $4$ & $2$ & N/A
\end{tabular}
\caption{Gained primes $p$ of good reduction for modular Atkin--Lehner quotients $X_0^*(N)$.}\label{CM_points_table} 
\end{table}
\end{small}


\newpage

\begin{footnotesize}
\begin{table}[ht!]
\begin{tabular}{c|c|c|c||c}
$N$ & $D$ & $g$ & $g^*$ & $p$ good for $J^{*}$ and $X^{*}$ \\ \hline\hline
2 & 91 & 19 & 1 & 2  \\ \hline
2 & 123 & 21 & 1 & 2  \\ \hline
2 & 133 & 27 & 2 & 2  \\ \hline
2 & 141 & 23 & 1 & 2  \\ \hline
2 & 145 & 29 & 2 & 2  \\ \hline
2 & 155 & 31 & 1 & 2  \\ \hline
2 & 177 & 29 & 2 & 2  \\ \hline
2 & 187 & 41 & 2 & 2  \\ \hline
2 & 213 & 35 & 2 & 2  \\ \hline
2 & 217 & 45 & 3 & 2  \\ \hline
2 & 247 & 55 & 3 & 2  \\ \hline
2 & 259 & 55 & 3 & 2  \\ \hline
2 & 267 & 45 & 3 & 2  \\ \hline
2 & 301 & 63 & 4 & 2  \\ \hline
2 & 1155 & 121 & 1 & 2  \\ \hline
2 & 1365 & 145 & 2 & 2  \\ \hline
2 & 1995 & 217 & 3 & 2  \\ \hline
3 & 142 & 23 & 1 & 3  \\ \hline
3 & 145 & 37 & 2 & 3  \\ \hline
3 & 158 & 27 & 1 & 3  \\ \hline
3 & 205 & 53 & 3 & 3  \\ \hline
3 & 1430 & 161 & 1 & 3  \\ \hline
5 & 91 & 37 & 1 & 5  \\ \hline
6 & 65 & 49 & 1 & 3  \\ \hline
6 & 85 & 65 & 1 & 3  \\ \hline
6 & 91 & 73 & 2 & 2  \\ \hline
6 & 115 & 89 & 2 & 2  \\ \hline
6 & 133 & 109 & 3 & 2  \\ \hline
10 & 57 & 53 & 1 & 2  \\ \hline
10 & 77 & 89 & 3 & 2  \\ \hline
10 & 93 & 89 & 2 & 2  \\ \hline
14 & 51 & 65 & 2 & 2  \\ \hline
15 & 38 & 37 & 1 & 3  \\ \hline
21 & 46 & 57 & 1 & 3  \\ \hline
26 & 33 & 69 & 2 & 2  \\ \hline
34 & 21 & 53 & 1 & 2  \\ \hline
34 & 35 & 109 & 3 & 2  \\ \hline
34 & 39 & 109 & 3 & 2  \\ \hline
46 & 15 & 49 & 1 & 2  \\ \hline
46 & 21 & 73 & 2 & 2  \\ \hline
51 & 14 & 37 & 1 & 3  \\ \hline
58 & 21 & 89 & 3 & 2  \\ \hline
74 & 15 & 77 & 2 & 2  \\ \hline
85 & 14 & 53 & 1 & 5  \\ \hline
93 & 14 & 65 & 1 & 3  \\ \hline
141 & 10 & 65 & 1 & 3 
\end{tabular}
\caption{Gained primes $p$ of good reduction for Shimura curves $X_0^*(D;N)$.}\label{Shimura_points_table} 
\end{table}
\end{footnotesize}

\bibliography{biblio}
\bibliographystyle{alpha}

\end{document}